\documentclass[11pt]{article} % use larger type; default would be 10pt

\usepackage[utf8]{inputenc} % set input encoding (not needed with XeLaTeX)

\usepackage{geometry} % to change the page dimensions
\usepackage{graphicx} % support the \includegraphics command and options

\usepackage{booktabs} % for much better looking tables
\usepackage{array} % for better arrays (eg matrices) in maths
\usepackage{verbatim} % adds environment for commenting out blocks of text & for better verbatim
\usepackage{subfigure} % make it possible to include more than one captioned figure/table in a single float
\usepackage{placeins}
\usepackage{amsmath,amsthm,amssymb} % standard math stuff
\usepackage{mathrsfs} % fancy fonts
\usepackage{mathtools}
\usepackage[utf8]{inputenc}
\usepackage[T1]{fontenc}
\usepackage{geometry}
\usepackage{amssymb,mathtools}
\usepackage{amsfonts}

\usepackage{fancyhdr} % This should be set AFTER setting up the page geometry
\usepackage{sectsty}
\allsectionsfont{\sffamily\mdseries\upshape} % (See the fntguide.pdf for font help)
\newgeometry{margin=2.5cm}

\usepackage[nottoc,notlof,notlot]{tocbibind} % Put the bibliography in the ToC
\usepackage[titles,subfigure]{tocloft} % Alter the style of the Table of Contents

\title{Canonical Algebraic Curvature Tensors of Symmetric and Anti-Symmetric Builds}
\author{Elise McMahon}

\newtheorem{theorem}{Theorem}[section]
\newtheorem{lemma}[theorem]{Lemma}

\theoremstyle{definition}
\newtheorem{conj}[theorem]{Conjecture}
\newtheorem{defn}[theorem]{Definition}

\newenvironment{definition}[1][Definition]{\begin{trivlist}
\item[\hskip \labelsep {\bfseries #1}]}{\end{trivlist}}

\newenvironment{remark}[1][Remark]{\begin{trivlist}
\item[\hskip \labelsep {\bfseries #1}]}{\end{trivlist}}

\usepackage{tikz}
\usepackage{verbatim}
\usetikzlibrary{arrows,chains,matrix,positioning,scopes}
\usepackage[small,nohug,heads=vee]{diagrams}
\diagramstyle[labelstyle=\scriptstyle]

\usepackage{graphicx}
\begin{document}
\maketitle

\abstract{ We relate canonical algebraic curvature tensors that are built from a self-adjoint ($R^S_A$) or skew adjoint ($R^{\Lambda}_A$) linear operator A.  Several authors have proven that any algebraic curvature tensor $R$ may be expressed as a sum of $R^S_A$, or as a sum of $R^{\Lambda}_A$. This motivates our interest in relating them as well as in the linear independence of sets of canonical algebraic curvature tensors. We develop an identity that relates $R^{\Lambda}_A$ to $R^S_A$, which will allow us to employ previous methods used for $R^S_A$ to the case of $R^{\Lambda}_A$ as well as use them interchangeably in some instances. We compute the structure group of $R^{\Lambda}_A$, and develop methods for determining the linear independence of sets which contain both $R^{\Lambda}_A$ and $R^S_A$. We consider cases where the operators are arranged in chain complexes and find that this greatly restricts the linear independence of the curvature tensors with those operators. Moreover, if one of the operators has a nontrivial kernel, we develop a method for reducing the bound on the least number of canonical algebraic curvature tensors that it takes to write a canonical algebraic curvature tensor. }

\section{Introduction and Motivation}

The set of algebraic curvature tensors over $\mathbb{R}$ is a vector space, denoted $\mathcal{A}(V)$. By Nash's Imbedding Theorem, an algebraic curvature tensor with respect to a symmetric canonical curvature tensor is realizable as the curvature tensor of an embedded hypersurface in Euclidean space. Thus, there is  interest in  $\mathcal{A}(V)$ and the symmetric canonical algebraic curvature tensors. Gilkey and Fiedler \cite{Fiedler1, GilkeyH}  proved that
\begin{equation} \label{eq:span}
\mathcal{A}(V)=\text{span} \{ R_A^S | \text{ for } A=A^{*} \} = \text{span} \{R_B^{\Lambda} | \text{ for } B=-B^* \}
\end{equation}
which motivates the study of linear independence of canonical algebraic curvature tensors. Previous results were concerned with sets of only symmetric canonical algebraic curvature tensors (denoted $R^S_A$) or only anti-symmetric canonical algebraic curvature tensors (denoted $R^{\Lambda}_B$ ) \cite{DiazDunn, DiazGarcia, Diroff}. Since both $R^S_B$ and $R^{\Lambda}_A$ span $\mathcal{A}(V)$, we are interested in the linear independence of sets containing both $R^S_B$ and $R^{\Lambda}_A$ and in how $R^S_B$ relates to $R^{\Lambda}_A$. By developing a new identity, we relate $R^{\Lambda}_A$ to $R^S_B$ and use it in considering the linear independence of sets containing both $R^S_B$ and $R^{\Lambda}_A$.

\begin{defn}
Let $V$ be a finite dimensional real vector space of dimension $n$. An \emph{ algebraic curvature tensor} $R$ is a multilinear map $R:  \otimes^4 V  \to \mathbb{R}$ such that 
\begin{enumerate}
\item $R(x,y,z,w)=-R(y,x,z,w)=R(z,w,x,y), $
\item $R(x,y,z,w)+R(z,x,y,w)+R(y,z,x,w)=0$.
\end{enumerate}
The last equation is the Bianchi Identity. 
\end{defn}

For more on algebraic curvature tensors see Gilkey \cite{GilkeyB, GilkeyH}. Let $\phi$ be a positive definite bilinear form throughout the paper and $V$ a finite dimensional vector space. Also, we use capital roman letters to denote linear endomorphisms of $V$.  Let $A^*$ denote the adjoint of $A$ with respect to $\phi$, characterized by $\phi (Ax, y)= \phi (x, A^* y)$. If a $A$ is stated to be self-adjoint or skew-adjoint, it is assumed tone with respect to $\phi$.
%If $\psi$ is a symmetric bilinear form on $\phi$, then $\psi (x,w) \psi (y,z) - \psi (x,z) \psi (y,w)$ is an algebraic curvature tensor. We write $R^S_{\psi}$. We call these the canonical algebraic curvature tensors. Similarly, if $\tau$ is an anti-symmetric bilinear form on $\phi$, then $\tau (x,w) \tau (y,z) - \tau (x,z) \tau (y,w)- 2 \tau (x,y) \tau(z,w)$ is an algebraic curvature tensor, we write $R^{\Lambda}_{\tau}$, and we call these the canonical algebraic curvature tensors.

\begin{defn}
Let $A$ and $B$ be a linear endomorphism of $V$ and $\phi$ the inner product. The \emph{symmetric build tensor with respect to $A$} is 
$$R^S_{A}(x,y,z,w)= \phi (Ax,w) \phi (Ay,z) - \phi (Ax,z) \phi (Ay,w).$$
The \emph{anti-symmetric build tensor with respect to B} is $$R^{\Lambda}_{B} (x,y,z,w) =\phi (Bx,w) \phi (By,z) - \phi (Bx,z) \phi (By,w) -2 \phi (Bx,y) \phi (Bz,w).$$
\end{defn}

$R^S_ {A}$ and $R^{\Lambda}_{B}$ satisfy the first equation in Definition \ref{eq:span}. $R^S_ {A} \in \mathcal{A}(V)$ if and only if $A = A^*$ \cite{GilkeyB}, \cite{DiazDunn}. If rank$(B) >2$,$R^{\Lambda}_{B} \in \mathcal{A}(V)$ if and only if $B= -B^*$  \cite{DiazDunn}, \cite{Diroff}. Then when $A=A^*$, $R^S_ {A}$ is called canonical and when $B=-B^*$, $R^{\Lambda}_{B}$ is called canonical.  Unless explicitly stated, we will assume that $R^S_A$ and $R^{\Lambda}_B$ are canonical algebraic curvature tensors.

We develop an identity that relates $R^{\Lambda}_A$ and $R^S_B$,  in Theorem  \ref{id:1}. We use that identity to prove Theorem \ref{le:3.2},  that for invertible $A:V \to V$, $A$ preserves $R^{\Lambda}_{B}$ under precomposition  if and only if $A$ preserves the bilinear form determined by $B$, up to a sign. Thus the structure groups of $R^{\Lambda}_B$ and $B$, denoted $G_{R^{\Lambda}_B}$ and $G_{B}$ respectively, are equivalent up to a sign.

This relates to results of Dunn, Franks, and Palmer, regarding the structure group of symmetric canonical curvature tensors, denoted $G_{R^S_{A}}$ \cite{DunnFranksPalmer}. The manner in which the symmetric and antisymmetric cases differ is interesting, as  $G_{R^S_{A}}= G_A$ if the signature of the inner product is balanced and $G_{R^S_{A}}= G^{\pm}_A$ only when the signature of the inner product is balanced.

Diaz and Dunn determine that under certain assumptions, commutative of the operators of the canonical algebraic curvature tensors is a necessary condition for the linear dependence of three symmetrically built  canonical curvature tensors, where one is with respect to the identity \cite{DiazDunn}. In Section 4, the identity that we developed allowed us to extend their result on commutivity to sets that contain both symmetric and anti-symmetric canonical algebraic curvature tensors. 

We call a set of algebraic curvature tensors \emph{ properly linearly dependent} if none of its proper subsets are linearly dependent. It suffices to consider proper linear dependence because in our cases the proper subsets have been already shown to be linearly dependent.
%The notion of diagonalization of self-adjoint operators translates to block diagonalization of $2 \times 2$ skew-symmetric blocks down the diagonal with zeros elsewhere for skew-adjoint operators. As this is the only block-diagonalization we consider, we will refer to it simply as\emph{ skew-block diagonalization}. 

We complete the classification of sets of three canonical algebraic curvature tensors (see \cite{Diroff} and  \cite{DiazDunn} for two cases). We prove that given basic rank conditions, $\{R^S_{I}, R^S_{B}, R^{\Lambda}_{C} \}$ and $\{R^S_{I}, R^{\Lambda}_{C}, R^{\Lambda}_{D} \}$ are linearly independent if respectively $\{ I, B, C \}$ and $\{I, C, D\}$ are linearly independent. Since there exist nontrivial $A$, $B$ symmetric such that $\{ I,A,B\}$ is linearly independent and $\{R^S_I, R^S_A, R^S_B \}$ is linearly dependent \cite{DiazDunn}, the above results indicate that sets of the same build are optimal for minimally expressing $R$ as a sum of canonical algebraic curvature tensors. 

To contrast the hypothesis of full rank which is usually seen \cite{DiazDunn},  we consider sets of canonical algebraic curvature tensors where any of the operators are allowed to have nontrivial kernels in Section 5. We consider cases where the operators of the canonical curvature tensors form a chain complex. Interestingly, the results hold whether each canonical curvature tensor is the symmetric or anti-symmetric build.

For example, if $A$, $B$, and $C$ are each symmetric or anti-symmetric linear endomorphism of $V$, in the following chain complex, and $aR_A \pm bR_B \pm cR_C=0$ for $a,b,c \in \mathbb{R}$ and nonzero, then $\{A, B, C\}$ is linearly dependent. Moreover, if $\emph{Rank }(A) \ge 4$ and $\emph{Rank }(C) \ge 4$, then  $C= \pm A$, and  $\delta =-1$. Furthermore, if the chain complex is an exact sequence and $B=-B^*$, then $A$ and $C$ are invertible.

\vspace{-5mm}
\begin{center}
\begin{diagram}
V &\rTo^{A} & V & \rTo^{B} & V & \rTo^{C} & V
\end{diagram}
\end{center}
We consider a more general arrangement of the operators, where $A, {B_1}, ..., {B_k}$ are arranged as $k$ sets of chain complexes each of length two, so either $\emph{Im } A \subseteq \emph{ker} B_i$ for all $i$, or $\emph{ker} A \subseteq \emph{Im } B_i $ for all $i$.  Then if $A, {B_1}, ..., {B_k}$ are each symmetric or anti-symmetric linear operators and  $0= aR_A+ \sum^k \pm b_i R_{B_i}$ for $a, b_i \in \mathbb{R}$ nonzero, then $R_A=0$. Moreover, if $A=-A^*$ then for each sequence that is exact, the corresponding $B_i$ is invertible. 

The linear independence of sets of four canonical curvature tensors has not been considered. We consider a set of four symmetric or anti-symmetric canonical curvature tensors, where the operators are arranged in a chain complex and it is very restrictive. 

The maximum number of $R^S_A$ required to write any $R$ as a sum of $R^S_A$ in a given $V$ of dimension $n$, is denoted $\nu(n)$  \cite{DiazGarcia} (for $R_B^{\Lambda}$, the number is denoted by $\eta(n)$ \cite{Diroff}). The numbers $\nu (R)$ and $\eta (R)$ provide better lower bounds for $\nu(n)$ and $\eta(n)$, defined as follows:

\begin{defn} Let $R$ denote an algebraic curvature tensor. Then 
\vspace{-5mm}
\begin{center}
\item $\nu(R) = min \{ k |   R= \sum^k_{i=0} a_i R_{A_i} \text{ for } a_i \in \mathbb{R} \text{ and } A_i=A_i^* \},$ and
\item $\eta(R) = min \{ k |  R= \sum^k_{i=0} a_i R_{B_i} \text{ for } a_i \in \mathbb{R} \text{ and } B_i=-B_i^* \}.$
\end{center}
\end{defn}

Letting $dim(V) =n$, $$\nu(n) := sup_{R \in \mathcal{A}(V)} \nu(R) \text{   \hspace{5mm} and   \hspace{5mm}  }  \eta(n) := sup_{R \in \mathcal{A}(V)} \eta(R).$$
Diaz-Ramos and Garcia-Rio prove that for $\emph{dim}V = n$, $\nu (n) \le n(n+1)/2$ \cite{DiazGarcia}.  Although the dimension of $\mathcal{A}(V)$ is $\frac{1}{12} n^2 (n^2-1)$, the bound is still far from optimal. The authors prove that for $n=3$, an algebraic curvature tensor requires at most two symmetric canonical algebraic curvature tensors to express it \cite{DiazGarcia}. Thus, there is interest in further reducing this bound. 

Our interest in relating $R^{\lambda}_B$ and $R^S_B$ is motivated by an interest in relating $\eta(R)$ and $\nu (R)$. One approach for doing this is by developing methods for reducing $\nu (R^{\Lambda}_A)$ and $\eta(R^S_A)$, given that one of the operators has a nontrivial kernel (in section 6). Each algebraic curvature tensor may be symmetric or antisymmetric build, and so written without superscript. Consider $R_C =\pm a R_B + \sum^k  \pm a_i R_{B_i}$, where $a, a_i \in \mathbb{R}$ and $\emph{ker}( B ) \ne 0$. If $A:V \to \emph{ker} ( B$) such that $A * (Bx, By) = (Bx, By)$ , then $R_{C} = \sum^k \epsilon_i R_{A^* B_i A}$. Moreover, $R_{A^* B_i A} \in \mathcal{A}(V)$, for both $B_i = B_i^*$ and $B_i = - B_i^*$. Thus, the same canonical curvature tensor is re-expressed as a sum of canonical curvature tensors with one fewer terms. If we apply this method to a curvature tensor of one type, expressed as a sum of another type, then this method reduces $\eta (R^S_B)$ or $\nu (R^{\Lambda}_B)$. 

As a more general case, we do not require $A$ to preserve any of the operators. This provides a method for reducing $\nu(R)$ and $\eta (R)$, given that at least one of the operators has a nontrivial kernel. If $R= \epsilon R_B + \sum^k \epsilon_i R_{B_i} $, where $\emph{ker} ( B ) \ne 0$. Then, for $A: V \to \emph{ker}(\tau)$, $\bar{R}=A*R = \sum^k \epsilon_i R_{A^* B_i A}$. Moreover, $R_{A^* B_i A} \in \mathcal{A}(V)$, for $B_i =B_i^*$ or $B_i=-B_i^*$. 

In both cases, the kernels each $B_i$ or $A_i$ in the sum  of their terms are aligned, as they contain $\emph{ker}A$. Since these methods extend to sums of both builds of curvature tensors, it provides motivation for introducing a new bound, $\mu(R),$ which allows the sum to be of both symmetric and anti-symmetric canonical algebraic curvature tensors (defined in Section 6).

%We include a useful lemma by Gilkey, \cite{GilkeyH}
%\begin{lemma} \label{act}
%Let $A$ be a linear map. Then $R^S_A \in \mathcal{A}(V)$ if an only if $A=A^*$ and $R^{\Lambda}_A \in \mathcal{A}(V)$ if and only if $A=-A^*$.
%\end{lemma}

\section{An Identity Relating the Symmetric and Anti-symmetric Build Canonical Algebraic Curvature Tensors}

We develop an identity for an anti-symmetric canonical algebraic curvature tensor in terms of symmetric build tensors. First we include the following two lemmas for completeness.

\begin{lemma} \label{le:dd}
Let $\phi$ be the inner product and $A \in L(V)$. Then for all $x, y, z, w \in V$, $$R^S_A(x,y,z,w)= R^S_{\phi} (Ax, Ay, z,w)= R^S_{\phi} (x, y, A^* z, A^* w).$$
\end{lemma}

\begin{proof}
The proof is straightforward and can be found in \cite{DiazDunn}. 
\end{proof}

\begin{lemma} \label{tensors}
If $Rank(A) \ge 3$, then $R_A^{S} \in \mathcal{A}(V)$ and  if and only if $A=A^*$. Also, $R_A^{\Lambda} \in \mathcal{A}(V)$ if and only if $A=-A^*$.
 \end{lemma}
 \begin{proof}
 The proofs are straightforward and can be found in \cite{GilkeyH} and  \cite{Diroff}, respectively.
 \end{proof} 

\begin{lemma} \label{le:t}
Let $\phi$ be the inner product, $A \in L(V)$. Then $$R^{\Lambda}_A(x,y,z,w)= R^S_A (x,y,z,w) - 2 \phi(Ax,y) \phi (Az,w).$$
\end{lemma}

\begin{proof}
\begin{align*}
R^{\Lambda}_A(x,y,z,w)
& = \phi (Ax,w) \phi (Ay,z) - \phi (Ax,z) \phi (Ay,w) -2 \phi (Ax,y) \phi (Az,w) \\
& = R^S_{\phi} (Ax, Ay, z,w) -2 \phi (Ax,y) \phi (Az,w) \\
& = R^S_A(x,y,z,w) - 2 \phi (Ax, y) \phi (Az, w).
\end{align*}

\end{proof}
 We can now prove our main result of this section, the identity that is mentioned above:

\begin{theorem} \label{id:1}
Let $A=-A^*$.  Then,

\begin{align*}
R^{ \Lambda}_A(x,y,z,w)
& = 2R^S_A(x,y,z,w)+R_A^S(x,z,y,w)+R_A^S(x,w,z,y) \\
& = 2R_{\phi}^S (Ax, Ay, z,w) + R^S_{\phi} (Ax,Az,y,w) +R^S_{\phi} (Ax,Aw, z,y). 
\end{align*}
\end{theorem}

\begin{proof}

Since $A=-A^*$, $R^{\Lambda}_A$ is an algebraic curvature tensor, so we can use the Bianchi Identity on $R^{\Lambda}_A$, and in combination with Lemmas \ref{le:t} and \ref{le:dd}, 
\begin{align*}
 R_A^{\Lambda} (x,y,z,w)  
& =  -R_A^{\Lambda} (z,x,y,w)-R_A^{\Lambda} (y,z,x,w) \\
& =  -R_A ^{S} (z,x,y,w) +2 \phi (Az,x) \phi (Ay,w)  - R_{A} ^{S} (y,z,x,w) +2 \phi (Ay, z) \phi (Ax, w) \\
& =  2R^S_{\phi}(Ax,Ay,z,w) -R_{A}^S (z,x,y,w)-R_{A}^S(y,z,x,w)  \\
& = 2R^S_A (x,y,z,w) +R^S_A(x,z,y,w) +R_A^S(x,w,z,y)
\end{align*}
\end{proof}

\begin{remark}
For an arbitrary endomorphism of $V$, $A$, $R^S_A(x,y,z,w)$, $R_A^S(x,z,y,w)$, and $R_A^S(x,w,z,y)$ are the symmetric build tensors; they are not necessarily algebraic curvature tensors by Lemma \ref{tensors}.
\end{remark}

\section{The Structure Group of $R_{\tau}$ }

In this section, we examine the relationship between a canonical algebraic curvature tensor and the corresponding bilinear form. Following the notation of \cite{DiazDunn}, we let $A*$ refer to precomposition with $A$, so $A*R_{\psi} = R_{\psi} (Ax, Ay, Az, Aw)$. 
Also, for a bilinear form  $\psi$,  $A* \psi = \psi(Ax, Ay)$. Let $GL(V)$ refer to the general linear group.

\begin{lemma} \label{le:1}

 Let $C=C^*$ and $B=-B^*$, then 
 \begin{equation} \label{eq:1}
  R^S_{C} (Ax, Ay, Az, Aw) = R^S_{A^* C A} (x,y,z,w)
  \end{equation}
  \begin{equation} \label{eq:2}
  R^{\Lambda}_{B} (Ax,Ay,Az,Aw)=R^{\Lambda}_{A^* B A}(x,y,z,w).
\end{equation}
\end{lemma}
 
\begin{proof}
% Then, by Lemma \ref{le:dd}
 %\begin{align*}
%R^S_{C} (Ax, Ay, Az, Aw)
%&= R^S_{\phi} (C Ax, C Ay, Az, Aw) \\
%&= R^{S}_{\phi} (A^* C A x, A^* C A y, z,w) \\
%&= R^S_{A^* C A} (x,y,z,w). 
%\end{align*}
  The proof  for \eqref{eq:1} is straightforward and can be found in \cite{DiazDunn}. For \eqref{eq:2} we use Theorem  \ref{id:1} and are then able to use the relations between the symmetric build curvature tensors and their operators. 

 Let $B=-B^*$. Then,
\begin{align*}
A*R^{\Lambda}_{B}
&= R^{\Lambda}_{B} (Ax,Ay,Az,Aw) \\
& = 2R^{S}_{B} (Ax, Ay, Az, Aw) + R^S_{B} (Ax, Az, Ay, Aw) + R^S_{B} (Ax, Aw, Az, Ay) \\
& = 2R^S_{\phi} (B Ax, B Ay, Az, Aw) + R^S_{\phi} (B Ax, B Az, Ay, Aw) + R^S_{\phi} (B Ax, B Aw, Az, Ay) \\
& = 2 R^S_{\phi} ( A^* B Ax, A^* B Ay, z,w) + R^S_{\phi} (A^* B A x, A^* B A z, y, w) + R^S_{\phi} (A^*  B A x, A^* B A w, z, y) \\
& = 2 R^S_{ A^* B A} (x,y,z,w) + R^S_{A^* B A} (x,z,y,w) + R^S_{ A^* B A} (x,w,z,y) \\
& =  R^{\Lambda}_{A^* B A} (x,y,z,w). 
\end{align*}
\end{proof}

For all $\phi(Ax,y)$, where $A$ is a symmetric linear operator, there exists a symmetric bilinear form $\psi$, such that $\psi (x, y) = \phi (Ax, y)$. Likewise, for all $\phi(Bx, y)$, where $B$ is an anti-symmetric linear operator, there exists an anti-symmetric bilinear form $\tau$, such that $\tau (x, y) = \phi (Bx, y)$. Then $R^S_A$ and $R^{\Lambda}_B$ are equivalent to $R^S_{\psi}$ and $R^{\Lambda}_{\tau}$, respectively. 

In the rest of this section we use $R^{\Lambda}_{\tau}$, where $\tau$ is an anti-symmetric bilinear form. 

\begin{definition}  Let $A \in GL(V^*)$ and let $\tau$ be an anti-symmetric bilinear form. The structure groups of $R_{\tau}$ and $\tau$ are $$ G_{R^{\Lambda}_{\tau}} = \{ A | A*R^{\Lambda}_{\tau} = R^{\Lambda}_{ \tau} \},$$ $$G_{\tau}= \{A | A* \tau = \tau \} \text{, and }$$  $$G^{ \pm }_{\tau}=\{A | A* \tau = \pm  \tau \}.$$
\end{definition}

\begin{theorem} \label{le:3.2}
For $\tau$ an anti-symmetric bilinear form with $Rank \tau \ge 4$, $G_{R^{\Lambda}_{\tau}}=G^{ \pm}_{\tau}$.
\end{theorem}

\begin{proof}
From Lemma \ref{le:1}, $A* R^{\Lambda}_{\tau}=R^{\Lambda}_{A* \tau}.$ Thus, if $A* \tau = \pm \tau$, then $A* R_{\tau} = R_{\tau}$ and so $G^{\pm}_{\tau} \subseteq  G_{R_{\tau}}$. 
Now, let $A \in G_{R^{\Lambda}_{\tau}}$ and so $A*R^{\Lambda}_{\tau} = R^{\Lambda}_{\tau}$. Then $$ R^{\Lambda}_{\tau}= A*R^{\Lambda}_{\tau} = R^{\Lambda}_{A* \tau}.$$ We apply a result of Gilkey \cite{GilkeyH},  
that $R^{\Lambda}_{A* \tau }= R^{\Lambda}_{\tau}$ implies that $A* \tau = \pm \tau$, giving the containment $G_{R_{\tau}} \subseteq R^{\pm}_{\tau}$.

\end{proof}

It is interesting to compare this result with the case of a symmetric bilinear form $\psi$. Dunn, Franks and Palmer \cite{DunnFranksPalmer} proved that for Rank  $\psi \ge 2$ , then $G_{R_{\psi}} = G_{\psi}$ if the signature of $\psi$ is imbalanced and $G_{R_{\psi}} = G_{\psi}^{\pm}$ if the signature of $\psi$ is balanced.  Thus, our result (which is independent of the signature of the inner product) matches the symmetric case in the more rare situation where the signature of $\phi$ is balanced. In the case where the signature of $\phi$ is imbalanced, our result has an the extra sign ambiguity. The sign ambiguity occurs independent of the inner product because for any anti-symmetric bilinear form $\tau$, there exists a linear operator $A$, such that $A* \tau = -\tau$, independent of the signature of the inner product.

\section{Proper Linear Dependence of $\{R^S_{\phi}, R^S_{\psi}, R^{\Lambda}_{\tau} \}$ }

The linear dependence of a set of three canonical algebraic curvature tensors which are all symmetric build has been addressed by Diaz and Dunn \cite{DiazDunn}. They determined that if $\phi$ is a positive definite symmetric bilinear form, $\psi,$ and $\tau$ are symmetric linear operators,  $ \tau$ full rank, Rank $( \psi ) \ge 3$, and $\{ R^S_{\phi}, R^S_{\psi}, R^S_{\tau } \}$ linearly dependent, then $\psi \tau = \tau \psi$. We consider replacing $R^S_{\tau}$ with $R^{\Lambda}_{\tau}$ with $\tau$ skew-adjoint,  and show that commutativity of $\psi$ and $\tau$ is still a necessary condition for the linear dependance of three canonical algebraic curvature tensors. 

Let \emph{proper linear independence} refer to the linear independence, where the subsets are assumed to be linearly independent. Since all cases of two canonical algebraic curvature tensors has already been considered (see \cite{DiazDunn} and \cite{Diroff}), we can use proper linear independence in place of linear independence without loss of generality.

We prove that $\{ R^S_I, R^S_B, R^{\Lambda}_C \}$  is linearly independent if $\{B, I, C \}$ is linearly independent, where $B=B^*$, $C= -C^*$, and $I$ refers to the identity. Since $\{ R^S_I, R^S_B, R^{S}_C \}$  is linearly dependent for certain $B$, and $C$ \cite{DiazDunn}, our results indicate that sets of only symmetric build canonical algebraic curvature tensors will provide a minimal expression of an algebraic curvature tensor.  

Sets of two canonical algebraic curvature tensors has previously been considered, and shown to be linearly independent if the corresponding operators are nonzero and linearly independent. We include these results in the following two lemmas, as we will use them in Sections 5 and 6.

 The proper linear dependence implies that the corresponding operators are nonzero and cannot be a scalar multiple of each other, because for $\lambda \in \mathbb{R}$, $B = \lambda A$ implies that $R_A+R_B = R_A +R_{\lambda A} = R_A + \lambda^2 R_A.$   The consideration of $\sum_i c_i R_{\bar{A_i}}$ for $c_i \in \mathbb{R}$, can be simplified by letting $A_i = \sqrt{ |c_i|} \bar{ A_i}$, so $\sum_i c_i R_{\bar{A_i}} = \sum_i \epsilon_i R_{A_i}$, where $\epsilon_i =$ sign($c_i$).

\begin{lemma} \label{le:(1}
If $B= B^*$, $A= -A^*$, $\emph{Rank }(B) \ge 3$, and $A \ne 0$, then $R^{\Lambda}_A \ne \pm R^S_B$. 
\end{lemma}
\begin{proof}

For contradiction, suppose 
$R^{\Lambda}_A = \pm R^S_B$. 
Choose a basis such that  $B$ is diagonalized and let $ \lambda_i$ be the eigenvalues. Since $A$ is nonzero, there exists an $i$ and $j$ such that $A_{ij} \ne 0$. By permuting the basis vectors, we can obtain that $A_{12} \ne 0$. Similarly, there exists $i$ such that $\lambda_i \ne 0$. 

For $i \ne 1$ or $2$,  evaluate the hypothesis with  $(e_1, e_2, e_i, e_1)$ so $A_{1i} A_{12} =0$. Then $A_{12} \ne 0$ implies $A_{i1} =0.$
Now evaluate the hypothesis with $(e_1, e_i, e_i, e_1)$ and $(e_2, e_1, e_1, e_2)$, 
so $\lambda_1 \lambda_i = 3A^2_{1i}$ and $\lambda_2 \lambda_1 = 3A^2_{12}$. Then $\lambda_{i} \ne 0$ implies $\lambda_1 =0$. Then $0 = \lambda_1 \lambda_2 = 3A_{12}^2$ contradicts that $A_{12} \ne 0$. 

If $i =1$ (and the case for $i=2$ is similar), evaluate the hypothesis with $(e_2, e_1, e_j, e_2)$ and $(e_1, e_2, e_j, e_1)$ for $j \ne 1$ or $2$, so $A_{2j}A{12} = 0$ and $A_{1j} A_{12} =0$. Thus, $A_{2j} =0$ and $A_{1j} =0$. Then evaluate the hypothesis with $(e_1, e_j, e_j, e_1)$ so $\lambda_1 \lambda_j = 3A^2_{ij} =0$, so $\lambda_j =0$ for all $j \ne 1$ or $2$. This contradicts $rank(B) \ge 3$. 
\end{proof}

A similar proof can be found in \cite{Lovel} but with different assumptions.

\begin{lemma} \label{le:(2}
For $A = A^*$, $B= B^*$, $\emph{Rank }(A) \ge 4$, and $\{ A,B \}$ linearly independant, $R^S_A \ne - R^S_B$; likewise, for $C = -C^*$, $C = -D^*$, $C$ and $D$ nonzero, $R^{\Lambda}_C \ne -R^{\Lambda}_D$. 
\end{lemma}
\begin{proof} 
The the proof can be found in Diaz and Dunn \cite{DiazDunn} and Diroff \cite{Diroff}.
\end{proof}

\begin{lemma} \label{le:dia}
Let $B=B^*$, $C=-C^*$, and $\emph{dim}V \ge 3$.  If $ \{ R^S_{I}, R^S_{B} , R^{\Lambda}_{ C } \} $ is  linearly dependent, then $BC = CB$ and $\emph{Rank }(C) =2$. 
\end{lemma}
 
\begin{proof}

$\{ R^S_I, R^S_B, R^{\Lambda}_C \}$  linearly dependent implies that $R^S_{ I} \pm R^S_{B} \pm R^{\Lambda}_{C}=0$ with $R^S_{ B}$ and $R^{\Lambda}_{C}$ nonzero ( as noted above). Multiply by $-1$ if necessary, so that the first term is positive, and let $\epsilon$, $\delta \in \{-1,+1\}$, so
\begin{equation} \label{eq:hyp}
 R^{\Lambda}_{C}= \epsilon R^S_{B}+ \delta R^S_{I}. 
\end{equation}
 
Let $\{ e_1,...,e_n \}$ be an orthonormal basis that diagonalizes $B$ and let $\lambda_i$ be the $i$th eigenvalue of $B$. Let $C_{ij}$ refer to the $i$th row and $j$th column of the matrix representation of $C$. Note that $C_{ij} = - C_{ji}$ and $C_{ii} =0$.
Note that $R^S_{ B}$ and $R^{\Lambda}_{C}$ nonzero implies that $Rank C \ge 1$ and $B \ne 0$. Since $C \ne 0$, there exist $i,j$ such that $C_{ij} \ne 0$. By permuting the basis vectors so that $B$ is kept diagonalized, we may assume without loss of generality that $C_{12} \ne 0$.

Now we prove that the entries of $C$ are all zero, except $C_{12}=-C_{21}$. Evaluating $(e_1, e_2, e_k, e_1)$ into the Equation \ref{eq:hyp}, results in $C_{12} C_{1k} =0.$ Similarly, $(e_2, e_1, e_k, e_2)$ results in $C_{12} C_{2k} =0$. Thus, $C_{1k}=C_{2k} =0$. 
Then evaluating Equation \ref{eq:hyp} with $(e_1, e_2, e_i, e_k)$ results in $$C_{1k} C_{2 i} - C_{1i} C_{2k} -2 C_{12} C_{ik}=0.$$ Then $-2 C_{12} C_{ik}=0$ and so $C_{ik}=0$. 
Then the matrix representation of $C$ with respect to $\{ e_1,...,e_n \}$ is skew-block diagonalized, since the only non-zero entries are $C_{12} = -C_{21}$. Since $B$ is diagonalized with respect to  $\{ e_1,...,e_n \}$, $BC =CB$. 

\end{proof}

\begin{theorem}
Let $B=B^*$, $C =-C^*$, and $\emph{dim} V >3$. If $\{ I, B, C\}$ is linearly independent, then  $\{R^S_{I}, R^S_{B}, R^{\Lambda}_{C} \}$ is  linearly independent. 
\end{theorem}

\begin{proof}
 Choose a basis $\{e_1,...,e_n \}$ that diagonalizes $B$ and let $\lambda_i$ refer to the $i$th eigenvalue of $B$. For contradiction, assume for $R^S_{B}$ and $R^{\Lambda}_{C}$ nonzero,
 \begin{equation} \label{eq:hyp2}
 R^S_{I}+  \epsilon R^S_{B}= \delta R^{\Lambda}_{C}
 \end{equation} 
 where $B=B^*$, $C=-C^*$, $\epsilon , \delta \in \{-1,1\}$, and $ \{B, C, I \}$ is linearly independent.  

By Lemma \ref{le:dia}, $C_{12} \ne 0$, and the rest of the entries of $C$ are zero. Then, for all pairs of $(i,j) \ne (1,2)$ or $(2,1)$, evaluate Equation \ref{eq:hyp2} with $(e_i, e_j, e_j, e_i)$, $(e_i, e_k, e_k, e_i)$, and $(e_j, e_k, e_k, e_j)$. Then,
 $$ 1 + \epsilon  \lambda_i \lambda_j =0, $$
$$ 1+ \epsilon  \lambda_i \lambda_k = 0, $$ 
$$ 1 + \epsilon \lambda_j \lambda_k=0. $$
Then $\lambda_i \ne 0$ for all $i$. Subtracting each equation from the other yields $0= \lambda_i (\lambda_j - \lambda_k )$, $0= \lambda_k ( \lambda_i - \lambda_j)$, and $0= \lambda_j ( \lambda_k - \lambda_i)$.  Then since $\lambda_i \ne 0$ for all $i$, $\lambda_j = \lambda_k$, $\lambda_i = \lambda_j$, and $\lambda_k=\lambda_i$.  Thus, $\lambda := \lambda_i = \lambda_j = \lambda_k$ and so  $B = \lambda I$, a contradiction.
\end{proof}

\begin{remark}
For dim $V \le 3$, there exists $B$ and $C$ such that $\{ C, D, I \}$ is linearly independent and $\{ R^S_I, R^S_B, R^{\Lambda}_C \}$ is  linearly dependent. Consider  $\epsilon = 1$ and $\delta = -1$. Let the eigenvalues of $B$ be $\lambda_1 = \lambda_2= 2$, $\lambda_3 = \frac{1}{2}$, and the only nonzero entries of $C$ be $C_{12} =1$, and $C_{21}= -1$. 
\end{remark}

For the case of all symmetric canonical curvature tensors, see \cite{DiazDunn} and for the case of all anti-symmetric canonical algebraic curvature tensors, see \cite{Diroff}. To complete the results on sets of three canonical algebraic curvature tensors,  we consider $\{ R^S_I, R^{\Lambda}_C, R^{\Lambda}_D \}$. A similar theorem was proven by Lovell \cite{Lovell}, and we include the proof for completeness. 

\begin{theorem}
Let $C=-C^*$, $D=-D^*$, and $\emph{dim} V \ge 3$. If $\{C, D, I\}$ is linearly independent, then $\{ R^S_I, R^{\Lambda}_C, R^{\Lambda}_D \}$ is linearly independent.  
\end{theorem}

\begin{proof}
For contradiction, assume that there exist $C$ and $D$ such that $\{ C, D, I \}$ is linearly independent and $R^S_{I}= \epsilon R^{\Lambda}_{C} + \delta R^{\Lambda}_{D}$ where $\epsilon, \delta \in \{ 1, -1\}$, and $R^{\Lambda}_{C}$ and $R^{\Lambda}_{D}$ are nonzero. Choose a basis $\{ e_1, ..., e_n \}$ so that $C$ is skew-block diagonalized. Then, in the matrix representation of $C$,  $C_{13}=0$ and $C_{23}=0$. 

Consider evaluating $(e_1, e_3, e_3, e_1),$ $(e_2, e_3, e_3, e_2),$ and $(e_3, e_1, e_2, e_3)$ into $R^S_{I}= \epsilon R^{\Lambda}_{C} + \delta R^{\Lambda}_{D}$ yields 
$$ 0 = \delta 3 D_{13} D_{23}, $$
$$ 1= \delta 3 D^2_{13}, \text{ and } $$
$$1= \delta 3 D_{23}^2. $$ The last two equations imply that $D_{13} \ne 0$ and $D_{23} \ne 0$, which contradict the first. 
\end{proof}

\section{Chain Complex and Linear Dependence}

To contrast from the full rank assumption of previous results \cite{DiazDunn}, we allow any of the operators to have a non-trivial kernel.  In these results we need not distinguish whether the canonical curvature tensors are symmetric or anti-symmetric build, so we put no superscript on $R$.  For the rest of the paper, if either $R^{\Lambda}_B$ or $R^S_A$ may be used, we will denote the canonical algebraic curvature tensor without the superscript. 
We will consider the particular case where the operators in a chain complex.

\begin{lemma} \label{le:0}
If $\emph{Im} A \subseteq \emph{ker} B$ or $\emph{Im}B \subseteq \emph{ker}A$, $B= \pm B^*$, then $B*R_A=0$.
\end{lemma}
\begin{proof}
Let $B= \pm B^*$. Either $BA=0$ or $AB=0$, since $\emph{Im }A \subseteq \emph{ker}B$ or $\emph{Im }B \subseteq \emph{ker}A$. Apply Lemma \ref{le:1} and so,  
\[ B^*R^S_A = R^S_{B^*AB}  = R^S_{ \pm BAB}  =0, \] 
\[ B^*R^{\Lambda}_A = R^{\Lambda}_{ B^*AB} = R^{\Lambda}_{ \pm BAB} = 0.\] 
\end{proof} 

\begin{lemma} \label{le:2}
If $A = \pm A^*$, and $\emph{Rank }(A^k)=p$, then $\emph{Rank }(A)=p$. 
\end{lemma}
\begin{proof}
For $A=A^*$, diagonalize $A$ with respect to $\phi$. Then \[A=
\left( \begin{array}{cccccc}
\lambda_1 &  & 0  &  & & \\
  &  \ddots &&  \\
 &  &  \lambda_p &  & \\
  0 &    &  & 0 &  \\
 &  &&& \ddots 
\end{array} \right), \text{ so }
A^k= 
\left( \begin{array}{cccccc}
\lambda_1^k &  & 0  &  & & \\
  &  \ddots &&  \\
 &  &  \lambda_p^k &  & \\
  0 &    &  & 0 &  \\
 &  &&& \ddots 
\end{array} \right).
 \] 
 Then $\emph{Rank }(A^k)=p$ if and only if $\lambda_i^k \ne 0$ for $1 \le i \le p$. Thus, $\lambda_i \ne 
0$ for $1 \le i \le p$, and so $\emph{Rank }(A)=p$. 

For $A=-A^*$, block-diagonalize $A$ in $2 \times 2$ blocks down the diagonal and zeros elsewhere. Then, each $2 \times 2$ block of $A$, denoted $\tilde{A}$, is of the form \[ \tilde{A} = \left( \begin{array}{cc}
0 & \lambda_i \\
-\lambda_i & 0 \end{array} \right). \]
 For $k$ even, the $2 \times 2$ blocks of $A^k$ are of the form
\[ \tilde{A}^{k} = \epsilon \left( \begin{array}{cc}
\lambda_i^{k} & 0 \\
0 & \lambda_i^{k} \\
 \end{array} \right) \]
where $\epsilon= 1$ if $k =0mod4$, and $\epsilon=-1$ if $k = 2 mod4$. For $k$ odd, the $2 \times 2$ blocks are of the form 
\[ \tilde{A}^{k} = \epsilon \left( \begin{array}{cc}
0 & \lambda_i^{k} \\
- \lambda_i^{k} & 0 \end{array} \right), \]
where $\epsilon=1$ if $k = 1mod4$ and $\epsilon=-1$ if $k = 3 mod 4$.
Then $\emph{Rank }(A^k)=p$ if and only if $\lambda_i^{k} \ne 0$ for $1 \le i \le p$. This happens if and only if $\lambda_i \ne 0$ for $1 \le i \le p$. Thus $\emph{Rank }(A) = p$. 
\end{proof}

This theorem is previously established, we included the proof for completeness, and it may be found in Kaplansky's Linear Algebra and Geometry. 

\begin{lemma} \label{le:(3}
If $R^S_{B}, R^{\Lambda}_C \in \mathcal{A}(V)$, then for any linear operator $A$, $R^S_{A^*BA}, R^{\Lambda}_{A^*CA} \in \mathcal{A}(V)$.
\end{lemma}

\begin{proof}
We refer to a result of Gilkey that $R^S_A \in \mathcal{A}(V)$ if an only if $A=A^*$ and $R^{\Lambda}_A \in \mathcal{A}(V)$ if and only if $A=-A^*$ \cite{GilkeyH}. Since $R^S_B \in \mathcal{A}(V)$, $B=B^*$. Consider $(A^* B A)^*= A^* B^* A= A^* B A$. Thus, $R^S_{A^* BA} \in \mathcal{A}(V)$.  For  $R^{\Lambda}_C \in \mathcal{A}(V)$, $C=-C^*$. Consider $(A^* C A)^*= A^* C^* A= -A^* C A$. Thus, $R^{\Lambda}_{A^* CA} \in \mathcal{A}(V)$.  
\end{proof}

We now reach the main results of the section. We show that in cases where the operators are arranged in a chain complex, the linear dependence is very restricted. In our considerations, the set of canonical algebraic curvature tensors can be of any combination of symmetric or anti-symmetric. As a result, we will denote the canonical algebraic curvature tensor without a superscript and assume that the operators are self or skew- adjoint if the tensor is respectively symmetric or skew symmetric.

\begin{theorem}
If $A$, $B$, and $C$ are each symmetric or anti-symmetric linear operators in the following chain complex, and $R_A+ \epsilon R_B + \delta R_C=0$ for $\epsilon, \delta = \pm1$, then $\{A, B, C\}$ is linearly dependent. Moreover, if $\emph{Rank }(A) \ge 4$ and $\emph{Rank }(C) \ge 4$, then  $C= \pm A$, and  $\delta =-1$. Furthermore, if the chain complex is an exact sequence and $B=-B^*$ then $A$ and $C$ are invertible.
\end{theorem}
\vspace{-12mm}
\begin{center}
\begin{diagram}
V &\rTo^{A} & V & \rTo^{B} & V & \rTo^{C} & V
\end{diagram}
\end{center}

\begin{proof}
By hypothesis, $R_A + \epsilon R_B + \delta R_C=0$ for $R_A, R_B,$ and $R_C$ symmetric or anti-symmetric canonical algebraic curvature tensors and $\epsilon, \delta = \pm 1$. The chain complex implies that $CB=0$ and $BA=0$. Then precomposing the sum with $B$ we obtain 
$$B* R_A (x, y, z, w)+ \epsilon B*R_B (x, y, z, w) +\delta B*R_C (x, y, z, w)=0.$$  
Applying Lemma \ref{le:0} to the above equation, results in $B*R_B (x, y, z, w)=0$. Lemma \ref{le:1} implies that $R_{B^3} (x,y,z,w)=0.$ Now we consider $B$ symmetric and anti-symmetric separately, so suppose $B=-B^*$. Then $R^{\Lambda}_{B^3}=0$ if and only if $B^3=0$ \cite{GilkeyH}. By Lemma \ref{le:2}, $B=0$. For the other case, suppose $B=B^*$, then $R_{B^3}^{S} =0$ if and only if $\emph{Rank }(B^3) \le 1$ \cite{GilkeyH}. By Lemma \ref{le:2}, $\emph{Rank }(B) \le 1$ and so $R^{S}_B=0$. Thus,  $R_B=0$   \cite{GilkeyH}. 
aaa

 Now add the assumption that $\emph{Rank }(A), \emph{Rank }(C) \ge 4$ and apply Lemma \ref{le:(1} to the resulting equation $R_A + \delta R_C=0$, so $R_A$ and $R_C$ must be the same build by \ref{le:(1}. Then $R^S_A = - \delta R^{S}_C$ or $R^{\Lambda}_A= - \delta R^{\Lambda}_C$, so by Lemma \ref{le:(2}, $\delta= -1$ and from \cite{GilkeyH}, $A= \pm C$. 

Finally, we add the assumptions that  $B=-B^*$ and the sequence is exact. Since $B=-B^*$, $R_B^{\Lambda}=0$ implies that $B=0$. Since the sequence is exact, then $\emph{Im }A=V$ and $\emph{ker}C=0$, so $A$ and $C$ are invertible.
\end{proof}

\begin{theorem}
If $A, {B_1}, ..., {B_k}$ are linear operators in one of the two following sets of chain complexes such that $0= R_A+ \sum^k \epsilon_i R_{B_i}$ where each $R$ symmetric or anti-symmetric build and $\epsilon_i \in \{-1, 1\}$, then $R_A=0$. Moreover, if $A=-A^*$ then for each sequence that is exact, then the corresponding $B_i$ is invertible. 
\end{theorem}
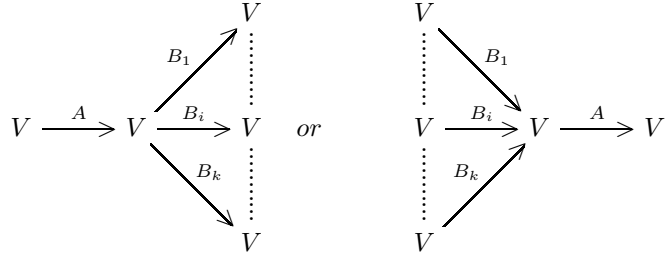
\hspace{-5cm} \begin{diagram}
& & &  & V  &&& V & & \\
& & & \ruTo^{B_1} & \dDots &&& \dDots & \rdTo^{B_1} & \\
V &  \rTo^{A} & V &\rTo^{B_i} &V & or && V &\rTo^{B_i} &V & \rTo^A & V \\
& & & \rdTo^{B_k} &\dDots &&& \dDots &\ruTo^{B_k} & \\
& & & & V &&& V & & \\
\end{diagram}

\begin{proof}
First, note that both diagrams depict a collection of $k$ chain complexes of length 2. Thus, $\emph{Im }A \subseteq \emph{ker}B_i$ so  $B_i A=0$ for all $i$, or $\emph{Im }B_i \subseteq \emph{ker}A$, so $AB_i =0$ for all $i$.
Consider $0= R_A + \sum \epsilon_i R_{B_i}$, where each canonical algebraic curvature tensor may be either symmetric or anti-symmetric build. Precompose the sum with $A$, so by Lemmas \ref{le:0} and \ref{le:1}, $$0= A* R_A + \sum \epsilon_i A* R_{B_i} = R_{A^3} + \sum \epsilon_i R_{AB_iA} = R_{A^3}.$$ Thus, $R_{A^3}=0$. For  $A=-A^*$, $R^{\Lambda}_{A^3}=0$, if and only if $A^3=0$ \cite{GilkeyH}. Then $A=0$ by Lemma \ref{le:2} so $R^S_A = 0$. If $A=A^*$, then $R^S_{A^3}=0$ if and only if $\emph{Rank } (A^3) \le 1$ \cite{GilkeyH}. Then $\emph{Rank } (A) \le 1$ by Lemma \ref{le:2}, and so $R_A^{S}=0$.

Now suppose that $A=-A^*$ and $\emph{Im }A=\emph{ker}B_i$ for some $i$ (or similarly, $\emph{Im }B_i = \emph{ker}A$), then $A=-A^*$, implies that $R_A$ is of anti-symmetric build, so $R^{\Lambda}_A =0$ implies that $A=0$. Then, $0=\emph{Im }A=\emph{ker}B_i$. Thus $B_i$ is invertible.
\end{proof}

\begin{theorem}
Let $A$, $B$, $C$, and $D$ be self or skew adjoint, and in the following chain complex. If $R_A + \epsilon_1 R_B + \epsilon_2 R_C + \epsilon_3 R_D=0$ and $\emph{Rank }(B) \ge 4$ and $\emph{Rank }(C) \ge 4$, then $R_A$ and $R_C$ are the same build and $R_B$ and $R_D$ are the same build. Moreover, $B^3 = \pm BDB$ and  $C^3 = \pm CAC$, $\epsilon_2 =-1$, and  $\epsilon_1 = - \epsilon_3$. 
\end{theorem}
\vspace{-10mm}
\begin{diagram}
V & \rTo^A & V  & \rTo^B & V & \rTo^C & V & \rTo^D & V \\
\end{diagram}

\begin{proof}

Precompose $R_A + \epsilon_1 R_B + \epsilon_2 R_C + \epsilon_3 R_D=0$ with $B$ and then separately with $C$, to get that 
\begin{equation} \label{eq:B2}
\epsilon_1 R_{B^3} + \epsilon_3 R_{BDB} =0, 
\end{equation}
\begin{equation} \label{eq:C2}
\epsilon_2 R_{C^3} + R_{CAC} =0.
\end{equation}

By Lemma \ref{le:(1} $R_{B^3}$ and $R_{BDB}$ must be of the same build, and so $R_B$ and $R_D$ must be the same build. Similarly, for $R_A$ and $R_C$. Applying Lemma \ref{le:(2} to Equation \ref{eq:B2} implies that $\epsilon_1= - \epsilon_3$, and to Equation \ref{eq:C2} implies $\epsilon_2=-1$. Finally, applying \cite{GilkeyH} results in that $B^3 = \pm BDB$ and $C^3= \pm CAC$.

\end{proof}

\section{Bounds on $\nu(R)$ and $\eta(R)$}

Diaz-Ramos and Garcia-Rio \cite{DiazGarcia} obtained an upper bound of $\frac{n(n+1)}{2}$ for $\nu(n)$, where $\nu(n) = \sup_{R \in \mathcal{A}(V) } \{ \nu (R) \}$; however this bound is far from optimal. Thus, we are interested in a general method for reducing $\nu(R)$. We develop a method for reducing the number of terms in a sum of canonical algebraic curvature tensors, given that at least one term has an operator with a nontrivial kernel.  Our interest in canonical algebraic curvature tensors, where the operators have nontrivial kernel is because the previous results have assumed the operators have full rank. Our methods in this section are general and also apply to $\eta(R)$, through the use of Theorem \ref{id:1}.

In this section, if either $R^{\Lambda}_B$ or $R^S_A$ may be used, we will denote the canonical algebraic curvature tensor without the superscript. 

\begin{theorem}
Let $R \in \mathcal{A}(V)$, $R= \epsilon R_{B} + \sum_{i=0}^k \epsilon_i R_{B_i} $,   where $\epsilon, \epsilon_i \in \{ \pm 1 \}$ and  $\emph{ker} ( B ) \ne 0$. Then for linear operator $A: V \to \emph{ker}(B)$, $\bar{R}=A*R = \sum_{i=0}^k \epsilon_i R_{A^* B_i A}$. Moreover, $R_{A^* B_i A} \in \mathcal{A}(V)$, and $R_{A^* B_i A}$ is the same build as $R_{B_i}$ for each $i$. 
\end{theorem}

\begin{proof}
Consider $R= \epsilon R_{B} + \sum_{i=0}^k \epsilon_i R_{B_i} $, where $\epsilon, \epsilon_i \in \{ \pm 1 \}$, ker $(B) \ne 0$, and $B_i = \pm B_i^*$. Let $A:V \to$ ker$( B )$. Then by Lemmas \ref{le:0} and \ref{le:1}, $$\bar{R} = A*R = \epsilon A* R_{B} + \sum_{i=0}^k \epsilon_i A* R_{B_i} = \sum_{i=0}^k \epsilon_i R_{A^* B_i A}. $$ 

By Lemma \ref{le:(3}, $R_{A^* B_i A} \in \mathcal{A}(V)$, and is the same build as $R_{B_i}$.
 
%Then, $R^S_{A^*B_i A} \in \mathcal{A}(V)$ if an only if $(A^* B_i A)^*=A^* B_i A$. If $R_{A^*B_i A}$ is symmetrically built, then $R_{B_i}$ must have been symmetrically built and so $B_i = B_i^*$. Thus   $(A^* B_i A)^*=A^* B_i A$.   For $R^{\Lambda}_{A^* B_j A} \in \mathcal{A}(V)$, if and only if $(A^* B_j A)^* =-A^* B_j A$. If $R_{A^* B_i A}$, is anti-symmetrically built, then $R^{\Lambda}_{B_i}$ must have been anti-symmetrically built and so $B_i=-B_i^*$. Thus,   $(A^* B_j A)^* =-A^* B_j A$.

\end{proof}
 If the curvature tensors are all of the same build, then this gives a method for reducing $\eta(R)$ or $\nu(R)$.

\begin{theorem} \label{th:6.2}
Consider $R_{C} = \epsilon R_{B} + \sum_{i=0}^k \epsilon_i R_{B_i}$, where $\epsilon, \epsilon_i \in \{ \pm 1 \}$ and $\emph{ker}( B ) \ne 0$. For linear operator $A:V \to \emph{ker} ( B )$ and $A* C = \pm C$, then $R_{C} = \sum_{i=0}^k \epsilon_i R_{A^* B_i A}$. Moreover, $R_{A^* B_i A} \in \mathcal{A}(V)$ and $R_{A^* B_i A}$ is the same build as $R_{B_i}$.
\end{theorem}

\begin{proof}
Consider $R_{C} = \epsilon R_{B} + \sum_{i=0}^k \epsilon_i R_{B_i}$, where $\epsilon, \epsilon_i \in \{ \pm 1 \}$ and $\emph{ker}( B ) \ne 0$. Let $A:V \to \emph{ker} (B)$, such that $A * C = \pm C$.  Then $A*R_C = R_C$ follows from Theorem  \ref{le:3.2}. Then, by Lemmas \ref{le:0} and \ref{le:1}, $$R_{C} = A* R_{C} = \epsilon A* R_{B} + \sum_{i=0}^k \epsilon_i A* R_{B_i}= \sum_{i=0}^k \epsilon_i R_{A^* B_i A}.$$ 

From Lemma \ref{le:(3}, $R_{A^* B_i A} \in \mathcal{A}(V)$ and remains the same build as $R_{B_i}$.

%If $R_{\gamma_i}$ is symmetrically built, then $\gamma_i^*=\gamma_i$ and so $(A^* \gamma_i A)^* = A^* \gamma_i A$. Thus, $R^{\Lambda}_{A^* \gamma_i A} \in \mathcal{A}(V)$. If $R_{\gamma_i}$ is anti-symmetrically built, then $\gamma_i^* = - \gamma_i$, and so $(A^* \gamma_i A)^* = - A^* \gamma_i A$, and so $R^S_{A^* \gamma_i A} \in \mathcal{A}(V)$.
\end{proof}
Regarding the requirement that the linear operator $A$ satisfy $A* C = \pm C$, if the consideration is parameterized by bilinear forms, then it is equivalent to $A$ being an isometry or para-isometry of the bilinear form characterized by $( Cx, y)$. 

%\begin{remark}
%In order for $A$ to preserve $C$, $A$ is of the form
%$ A=\left[
%\begin{array}{cc}
%\tilde{A} & 0 \\ 
%\bar{B} & \tilde{B}
%\end{array}\right]$, such that $\tilde{B}$ is invertible.
% Then $C$ is preserved because the block next to $\tilde{A}$ is 0. 
%In order for $A$ to map to the kernel of $B$, let $\tilde{A}= ker(B)$. Thus we have constructed an $A$ that Theorem %\ref{th:6.2}. 
%\end{remark}

This motivates a relationship between $\nu(R)$ and $\eta (R)$. In particular, if $R^{\Lambda}_A = \sum R^S_{B_i}$, then the theorem gives a method for reducing $\nu (R^{\Lambda}_{A})$ and for the opposite case, a method for reducing $\eta (R^S_{B})$. 
If the sums are combinations of both types of tensors, these theorems motivate the definition of a new bound, $\mu(R)$, and how to possibly reduce it.

\begin{definition}
Let $\mu (R) = min \{ k | R= \sum_{i=0}^k R_A, \text{ where } A =A^* \text{ or } A =-A^* \}$.
\end{definition}

Clearly $\mu(R) \le min \{ \nu(R), \eta(R) \}$ and  $\mu(n) \le min \{ \nu(n), \eta(n) \}$. Then, since $\nu(2)=1$ from \cite{GilkeyH}, we can conclude that $\mu(2)=1$.  Based on the case of sets of three canonical curvature tensors, we make the following conjecture: 

\begin{conj}
$\mu(R) = \nu(R)$
\end{conj}

I would like to thank Dr. Corey Dunn and Dr. Michael Marsalli for their helpful insights and guidance. I would also like to thank the NSF for funding.

%Include previous results: place them as lemmas in te section where they are needed
%Include lemma to replace Lovell's citation
%Include case of four tensors: 
%fix linear dependence
%find Diroffs results
%footnote Lovells result
%figure out A=I

\end{document}